\documentclass{amsart}
\usepackage{amssymb,amsmath,latexsym}
\usepackage{amsthm}
\usepackage{fontenc}

\numberwithin{equation}{section}
\newtheorem{theorem}{Theorem}[section]

\newtheorem{lemma}[theorem]{Lemma}

\begin{document}
\author{Alexander E Patkowski}
\title{On summatory arithmetic functions and a Volterra Integral equation}

\maketitle
\begin{abstract} We obtain asymptotic results for well known summatory arithmetic functions, such as $\psi(x),$ and establish connections to new summatory functions. A new Volterra integral equation is offered, which is solved by summatory arithmetic functions. We conclude with some further integral formulas and provide number theoretic formulas as applications.\end{abstract}

\section{Introduction and asymptotic formulas}  
A summatory arithmetic function is, generally speaking, of the form $\sum_{n\le x}a(n),$ where $a(n)$ is an arithmetic function $a(n):\mathbb{N}\rightarrow\mathbb{C}.$ In studying summatory arithmetic functions , it is desired to obtain information on its behavior when $x$ is large. Several famous results on this topic have a central place in the analytic theory of numbers, such as the Prime Number Theorem [6] ($a(n)=\Lambda(n)$ the von Mangoldt function), which states
\begin{equation}\sum_{n\le x}\Lambda(n)\sim x,\end{equation} as $x\rightarrow\infty.$ Here $f(x)\sim g(x)$ means that $\lim_{x\rightarrow\infty}f(x)/g(x)=1.$ Recall the defining property of an asymptotic expansion [2, pg.355, Property(A)] of a function $f(x)$ is
$$\lim_{x\rightarrow\infty}\left(x^N(f(x)-F_N(x))\right)=0,$$ where $F_N(x)=\sum_{0\le n\le N}b_nx^{-n}.$ See also [5, pg.179] for relevant material on asymptotic expansions by Mellin inversion.
\par The main purpose of this paper is to offer new results on summatory arithmetic functions, including asymptotics. In particular, we will show $\sum_{n\le x}a(n)$ solves a particular Volterra integral equation. 
\begin{theorem} Let $H_w(n)=\sum_{d|n}\Lambda(d)d^{-w},$ and put $w\ge0.$ We have, as $y\rightarrow\infty,$
$$\sum_{n\le y}\frac{\Lambda(n)}{n^{w}}\{\frac{y}{n}\}[\frac{y}{n}] -\sum_{n\le y}H_w(n)\sim \kappa(w)$$
$$+\frac{\zeta'(w)}{6\zeta(w)}-\frac{y\zeta'(1+w)}{2\zeta(1+w)}+\sum_{\rho}\frac{y^{\rho-w}(2+w-\rho)\zeta(\rho-w-1)}{2(\rho-w)(\rho-w-1)}+\sum_{n\ge1}\frac{y^{-2n-w}(2+w+2n)\zeta(-2n-w-1)}{2(2n+w)(2n+w+1)},$$
where $\kappa(0)=\kappa(1)=0,$ and otherwise $$\kappa(w)=-\frac{(1+w)y^{1-w}\zeta(-w)}{(w-1)w}.$$
\end{theorem}

\begin{proof} First, [1, pg. 526, Lemma 9] tells us that for $x>1,$
\begin{equation} \{x\}[x]=\frac{1}{2\pi i}\int_{(r)}\frac{x^{s}}{s(s-1)}\left((s-1)\zeta(s)+(2-s)\zeta(s-1)\right)ds,\end{equation}
where $r>1.$ Hence (putting $x=\frac{y}{n}$ in (1.2)),

\begin{equation} \begin{aligned}\sum_{n\le y}\frac{a(n)}{n^{w}}\{\frac{y}{n}\}[\frac{y}{n}] &=\frac{1}{2\pi i}\int_{(r)}\frac{y^{s}L(s+w)}{s(s-1)}\left((s-1)\zeta(s)+(2-s)\zeta(s-1)\right)ds\\
&=\sum_{n\le y}G_w(n)+\frac{1}{2\pi i}\int_{(r)}\frac{y^{s}L(s+w)}{s(s-1)}(2-s)\zeta(s-1)ds\end{aligned}
\end{equation}
for $w\ge0,$ where $G_w(n)=\sum_{d|n}a(d)d^{-w},$ since $L(s+w)\zeta(s)=\sum_{n\ge1}G_w(n)n^{-s}.$ Put $a(n)=\Lambda(n).$ Then the integrand has a simple poles at $s=0,$ $s=1,$ $s=1-w$ if $w$ is not $0$ or $1,$ $s=\rho-w,$ and $s=-2n-w.$ Note that 
$$\lim_{s\rightarrow1-w}\left((s+w-1)\frac{y^{s}\zeta'(s+w)}{s(s-1)\zeta(s+w)}(2-s)\zeta(s-1)\right)=-\frac{(1+w)y^{1-w}\zeta(-w)}{(w-1)w}.$$

$$\lim_{s\rightarrow0}\left(s\frac{y^{s}\zeta'(s+w)}{s(s-1)\zeta(s+w)}(2-s)\zeta(s-1)\right)=\frac{\zeta'(w)}{6\zeta(w)}.$$

$$\lim_{s\rightarrow1}\left((s-1)\frac{y^{s}\zeta'(s+w)}{s(s-1)\zeta(s+w)}(2-s)\zeta(s-1)\right)=-\frac{y\zeta'(1+w)}{2\zeta(1+w)}.$$

$$\lim_{s\rightarrow\rho-w}\left((s-\rho+w)\frac{y^{s}\zeta'(s+w)}{s(s-1)\zeta(s+w)}(2-s)\zeta(s-1)\right)=\frac{y^{\rho-w}(2+w-\rho)\zeta(\rho-w-1)}{2(\rho-w)(\rho-w-1)}.$$

$$\lim_{s\rightarrow-2n-w}\left((s+2n+w)\frac{y^{s}\zeta'(s+w)}{s(s-1)\zeta(s+w)}(2-s)\zeta(s-1)\right)=\frac{y^{-2n-w}(2+w+2n)\zeta(-2n-w-1)}{2(2n+w)(2n+w+1)}.$$ Collecting these residues gives the result.  \end{proof}
Note that this result is only valid asymptotically, as the sums diverge, due to the growth of the terms involving the negative real part of $\zeta(s).$ Possible uses of Theorem 1 for approximation could be achieved by truncating the two divergent sums. In the case $w=0,$ $H_0(n)=\log(n),$ the left hand side is then 
$$\sum_{n\le y}\Lambda(n)\{\frac{y}{n}\}[\frac{y}{n}] -\log(\Gamma(y+1)).$$ The interested reader may further analyze the growth of sum on the left hand side using Stirling's formula for $\Gamma(y).$

Let $p_j(x)$ denote a polynomial of degree $j.$ Since the sum over $\rho$ contains the ratio $p_{1}(\rho)/p_{2}(\rho),$ we investigate the convergence of a similar sum involving $p_{0}(\rho)/p_{1}(\rho).$ This simplifies our arguments while still achieving our objective of showing divergence.

\begin{lemma} Assume the Riemann Hypothesis. Then the sum
$$\sum_{\rho}\frac{x^{\rho}\zeta(\rho-M')}{\rho}$$ diverges for real numbers $M'\ge1.$
\end{lemma}
\begin{proof}First we assume $\Re(\rho)=\frac{1}{2}.$ If $\sigma<0,$ then it is known through the functional equation that [6, pg.95, eq.(5.1.1)] $|t|^{\frac{1}{2}-\sigma}\gg \zeta(\sigma+it)\gg |t|^{\frac{1}{2}-\sigma}.$ Hence $|\gamma|^{M'}\gg \zeta(\rho-M')\gg |\gamma|^{M'},$ for $M'>\frac{1}{2}$ as $\gamma\rightarrow\infty.$ To see this, note that if $\gamma_k=\Im(\rho_k)$ for the $k$th zero, then [6, pg.214, eq.(9.4.4)] $$\gamma_k\sim\frac{2\pi k}{\log(k)},$$ as $k\rightarrow\infty.$

In particular, for any integer $M'>\frac{1}{2},$
$$|\zeta(\frac{1}{2}-M'+i\gamma_k)|<C_1\left(\gamma_k\right)^{M'} \sim  C_1\left(\frac{2\pi k}{\log(k)}\right)^{M'},$$ for a positive constant $C_1,$ as $k\rightarrow\infty.$

For $M'>\frac{1}{2},$
\begin{equation} y^{1/2}\sum_{0<\gamma\le T}\frac{|\gamma|^{M'}}{\gamma}\ll \sum_{0<\gamma\le T}\frac{y^{\rho}\zeta(\rho-M')}{\rho}\ll y^{1/2}\sum_{0<\gamma\le T}\frac{|\gamma|^{M'}}{\gamma}.\end{equation}
By [6, Theorem 9.4], as $T\rightarrow\infty,$
\begin{equation}N(T):=\sum_{0<\gamma\le T}1\sim \frac{1}{2\pi}T\log(T).\end{equation}
If $M'=1$ then (4) is approximately $\sqrt{y}N(T),$ which by (5) tends to infinity when $T\rightarrow\infty.$ The result now follows for $M\ge1$ from the squeeze theorem. \end{proof}
Comparing the sum over $\rho$ in our theorem with Lemma 1, shows it is also divergent. Recall [6, pg.96, eq.(5.1.6)] that the Lindel$\ddot{o}$f Hypothesis states that $|t|^{\frac{1}{2}-\sigma}\gg \zeta(\sigma+it)\gg |t|^{\frac{1}{2}-\sigma}$ for $\sigma<\frac{1}{2}.$ Note that if we assume the Lindel$\ddot{o}$f Hypothesis, the sum in Lemma 1 diverges for $M'>0.$
\par It is interesting to observe that we may further simplify the integral we used in the following way.
\begin{equation} \begin{aligned}\sum_{n\le x}\frac{a(n)}{n^{w}}\{\frac{x}{n}\}[\frac{x}{n}] &=\frac{1}{2\pi i}\int_{(r)}\frac{x^{s}L(s+w)}{s(s-1)}\left((s-1)\zeta(s)+(2-s)\zeta(s-1)\right)ds\\
&=\sum_{n\le x}G_w(n)+\frac{1}{2\pi i}\int_{(r-1)}\frac{x^{s+1}L(s+1+w)}{s(s+1)}(1-s)\zeta(s)ds\\
&=\sum_{n\le x}G_w(n)+\int_{0}^{x}\left(\sum_{n\le y}G_{w+1}(n)-\frac{y^2}{2}L(2+w)\right)dy\\
&+\frac{1}{2\pi i}\int_{(r)}\frac{x^{s}L(s+w)}{s}\zeta(s-1)ds\\
&=\sum_{n\le x}G_w(n)+\int_{0}^{x}\left(\sum_{n\le y}G_{w+1}(n)-\frac{y^2}{2}L(2+w)\right)dy\\
&+\left(\sum_{n\le x}nG_{w+1}(n)-\frac{x^2}{2}L(2+w)\right).\end{aligned}
\end{equation}

\begin{theorem} Let $\mu(n)$ denote the M$\ddot{o}$bius function. As $y\rightarrow\infty$
$$-\sum_{n\le y}\mu(n)\log(n)\{\frac{y}{n}\}[\frac{y}{n}] -\psi(y)\sim \frac{y}{2}+\sum_{\rho}\bar{R}_{\rho}(y)y^{\rho}+\sum_{n\ge1}\bar{R}_{-2n}(y)y^{-2n}.$$\end{theorem}
\begin{proof}
If in (1.3) we choose $a(n)=-\mu(n)\log(n),$ we have that 
\begin{equation} \begin{aligned}-\sum_{n\le y}\mu(n)\log(n)\{\frac{y}{n}\}[\frac{y}{n}] &=\frac{1}{2\pi i}\int_{(r)}\frac{y^{s}\zeta'(s)}{\zeta^2(s)s(s-1)}\left((s-1)\zeta(s)+(2-s)\zeta(s-1)\right)ds\\
&=\sum_{n\le y}\Lambda(n)+\frac{1}{2\pi i}\int_{(r)}\frac{y^{s}\zeta'(s)}{\zeta^2(s)s(s-1)}(2-s)\zeta(s-1)ds\\
&=\sum_{n\le y}\Lambda(n)+\frac{1}{2\pi i}\int_{(r-1)}\frac{y^{s+1}\zeta'(s+1)}{\zeta^2(s+1)s(s+1)}(1-s)\zeta(s)ds\end{aligned}
\end{equation}
This last integral has poles at $s=0,$ $s=1,$ $s=\rho-1,$ and $s=-2n-1.$ We compute the residues as (and define a function $\bar{R}_z(y)$)

$$\lim_{s\rightarrow0}\left(s\frac{y^{s+1}\zeta'(s+1)}{\zeta^2(s+1)s(s+1)}(1-s)\zeta(s)\right)=\frac{y}{2}.$$
$$\lim_{s\rightarrow1}\left((s-1)\frac{y^{s+1}\zeta'(s+1)}{\zeta^2(s+1)s(s+1)}(1-s)\zeta(s)\right)=0.$$
$$y^{\rho}\bar{R}_{\rho}(y):=\lim_{s\rightarrow\rho-1}\frac{d}{ds}\left((s-(\rho-1))^2\frac{y^{s+1}\zeta'(s+1)}{\zeta^2(s+1)s(s+1)}(1-s)\zeta(s)\right)=$$
$$-\frac{y^{\rho}}{(\rho-1)^2\rho^2\zeta'(\rho)}\bigg(\rho^3\zeta'(\rho-1)-3\rho^2\zeta'(\rho-1)+2\rho\zeta'(\rho-1)$$
$$-\rho^2\zeta(-1+\rho)+4\rho\zeta(-1+\rho)-2\zeta(-1+\rho)+\log(y)\rho^3\zeta(-1+\rho)-3\log(y)\rho^2\zeta(-1+\rho)+2\log(y)\rho\zeta(-1+\rho)\bigg).$$
Combining these computations with Cauchy's residue theorem gives the result. \end{proof}
A direct consequence of this theorem is that 
\begin{equation}-\sum_{n\le y}\mu(n)\log(n)\{\frac{y}{n}\}[\frac{y}{n}]\sim \frac{3y}{2},\end{equation}
as $y\rightarrow\infty.$ This tells us that the growth order of the sum (1.8) is the same as $\psi(y).$
\par If we choose $a(n)=\mu(n),$ $w=0,$ in (1.3) and let $\phi(n)$ be the Euler totient function, we have for $x>1,$
\begin{equation}\sum_{n\le x}\mu(n)\{\frac{x}{n}\}[\frac{x}{n}] =1+\int_{0}^{x}\left(\sum_{n\le y}\frac{\phi(n)}{n}-\frac{6y^2}{2\pi^2}\right)dy+\left(\sum_{n\le x}\phi(n)-\frac{6x^2}{2\pi^2}\right).\end{equation}
The far right hand side of (1.9) may be recognized as the error term $E(x)$ for the summatory Euler totient function [3], which has been estimated in [7]. 
\par If we compute residue at $s=1,$ and $s=0$ of (1.2), we get
\begin{equation} \{y\}[y]=\frac{y}{2}-\frac{1}{3}+\frac{1}{2\pi i}\int_{(r')}\frac{y^{s}}{s(s-1)}\left((s-1)\zeta(s)+(2-s)\zeta(s-1)\right)ds,\end{equation}
for $-1<r'<0.$ This reduces to
\begin{equation} \{y\}[y]=\frac{y}{2}-\frac{1}{3}+\frac{1}{2}-\{y\}+\frac{1}{2\pi i}\int_{(r')}\frac{y^{s}}{s(s-1)}(2-s)\zeta(s-1)ds.\end{equation}
Using similar arguments it is possible to prove
\begin{equation}\begin{aligned} \{y\}[y]&=[y]+\frac{1}{2}+\frac{1}{2\pi i}\int_{(r)}\frac{y^{s}}{s(s-1)}(2-s)\zeta(s-1)ds\\
&=[y]+\frac{1}{2}+\frac{1}{2\pi i}\int_{(r-1)}\frac{y^{s+1}}{s(s+1)}(1-s)\zeta(s)ds\\
&=[y]+\frac{1}{2}-\frac{1}{2}(\{y\}^2+[y])+\frac{1}{2\pi i}\int_{(r-1)}\frac{y^{s+1}}{(s+1)}\zeta(s)ds\\ 
&=[y]+\frac{1}{2}-\frac{1}{2}(\{y\}^2+[y])+\frac{1}{2\pi i}\int_{(r)}\frac{y^{s}}{s}\zeta(s-1)ds\\ 
\end{aligned}\end{equation}
We can get $y=\frac{x}{n}>1$
$$\sum_{n\le x}\mu(n)\left( \{\frac{x}{n}\}[\frac{x}{n}]-[\frac{x}{n}]-\frac{1}{2}+\frac{1}{2}(\{\frac{x}{n}\}^2+[\frac{x}{n}])\right)=\sum_{n\le x}\phi(n)-\frac{3x^2}{\pi^2}.$$
\section{ A Volterra Integral Equation}
Recall that for a function $f(y)\in C[q_0,q_1],$ the Volterra integral equation of the second kind has the form [4, pg.41], $y\in[q_0,q_1],$
$$f(y)=g(y)+\int_{q_0}^{y}K(y,x)f(x)dx.$$
It is known that the solutions to the Volterra integral equation of the second kind are unique [4, pg.41, Theorem 3.10]. Furthermore, $f(y)$ possesses a convergent Neumann series representation [4, pg.196, Theorem 10.20] (see also [4, pg.193, Theorem 10.15]).

\begin{theorem} The Dirichlet polynomial $D_w(y):=\sum_{n\le y}a(n)/n^{w},$ for $w\ge1,$
is a solution to the Volterra integral equation
$$D_w(y)=F_w(y)+\frac{1}{y}\int_{0}^{y}D_w(y_0)dy_0,$$
where $F_w(y):=\sum_{n\le y}a(n)\{\frac{n}{y}\}/n^{w}.$ 
Furthermore, we have the Neumann-type series,
$$D_w(y)=F_w(y)+\frac{1}{y}\int_{0}^{y}F_w(y_0)dy_0+\sum_{k\ge1}\frac{1}{y}\int_{0}^{y}\int_{0}^{y_0}\cdots\int_{0}^{y_{k-1}}\frac{F_w(y_k)}{y_0y_1\cdots y_{k-1}}dy_0\cdots dy_k.$$

\end{theorem}

\begin{proof}
From [6, pg.15]
$$\int_{0}^{1}\left(\frac{1}{2}-\{x\}\right)x^{-s-1}dx=\frac{1}{s(s-1)}+\frac{1}{2s},$$ for $\Re(s)<0.$ This is equivalent to
\begin{equation}\left(\frac{1}{2}-\{x\}\right)=\frac{1}{2\pi i}\int_{(l)}\left(\frac{1}{s(s-1)}+\frac{1}{2s}\right)x^{s}ds,\end{equation} $l<0,$ for $0<x<1.$ Selecting $x=ny$ in (2.1), and inverting the desired sum gives, for $y>0,$
$$\sum_{n<1/y}\frac{a(n)}{n}\left(\frac{1}{2}-\{ny\}\right)=\frac{1}{2\pi i}\int_{(l)}\left(\frac{1}{s(s-1)}+\frac{1}{2s}\right)L(1-s)y^{s}ds.$$
Replacing $s$ by $1-s$ we have
$$\sum_{n<1/y}\frac{a(n)}{n}\left(\frac{1}{2}-\{ny\}\right)=\frac{1}{2\pi i}\int_{(1-l)}\left(\frac{1}{s(s-1)}+\frac{1}{2(1-s)}\right)L(s)y^{1-s}ds.$$
Replacing $s$ by $s+1,$ and replace $y$ by $1/y,$ we get (since $-l>0,$ and $L(s+1)$ is analytic for $\Re(s)>0$)
$$\begin{aligned}\sum_{n\le y}\frac{a(n)}{n}\left(\frac{1}{2}-\{\frac{n}{y}\}\right)&=\frac{1}{2\pi i}\int_{(-l)}\left(\frac{1}{s(s+1)}-\frac{1}{2s}\right)L(s+1)y^{s}ds\\
&=-\frac{1}{2}\sum_{n\le y}\frac{a(n)}{n}+\frac{1}{y}\int_{0}^{y}\left(\sum_{n\le t}\frac{a(n)}{n}\right)dt.\end{aligned}$$
We have proven that,
$$\sum_{n\le y}\frac{a(n)}{n}\left(1-\{\frac{n}{y}\}\right)=\frac{1}{y}\int_{0}^{y}\left(\sum_{n\le t}\frac{a(n)}{n}\right)dt.$$
Giving the Volterra integral equation with $f(y)=\sum_{n\le y}\frac{a(n)}{n},$

\begin{equation}f(y)=g(y)+\frac{1}{y}\int_{0}^{y}f(y_0)dy_0,\end{equation}
and by [4, pg.196, Theorem 10.20], we get the convergent Neumann series,
$$f(y)=g(y)+\frac{1}{y}\int_{0}^{y}g(y_0)dy_0+\sum_{k\ge1}\frac{1}{y}\int_{0}^{y}\int_{0}^{y_0}\cdots\int_{0}^{y_{k-1}}\frac{g(y_k)}{y_0y_1\cdots y_{k-1}}dy_0\cdots dy_k.$$
Replacing $a(n)$ by $a(n)n^{-w+1}$ gives the theorem.
\end{proof}

Note that the only solution to the homogeneous form of (2.2) is the constant function. This may be seen by noting that if $f(y)$ is a suitable analytic function, then it possesses a Taylor series with coefficients $b_n,$ say. Subsequently, after equating coefficients, the homogenous form ($g(y)=0$) implies  that the coefficients of $f(y)$ satisfy $b_n=(n+1)b_n.$ Therefore, $b_n=0$ for $n>0,$ and the solution is $f(y)=b_0.$ 
\par A similar integral equation was found in [3], however our kernel differs since $K(x,t)=1/x$ for $0<t\le x,$ while theirs has $K(x,t)=1/t$ for $0<t\le x.$
It is also possible to recast our integral equation in the form of a Boundary value problem with criteria for the Riemann Hypothesis. Namely, if $w\ge1,$
$$\frac{d}{dx}\left(xD_w(x)\right)+\frac{d}{dx}\left(xU_w(x)\right)=D_w(x),$$
where $U_w(x)=\sum_{n\le x} \frac{\mu(n)}{n^w}\{\frac{n}{x}\}.$ The solution being $\sum_{n\le x}\mu(n)n^{-w},$ is equivalent to the Riemann Hypothesis if we impose the "boundary condition" that $\sum_{n\le x}\mu(n)n^{-w}=O(x^{\frac{1}{2}-w+\epsilon}),$ for every $\epsilon>0,$ as $x\rightarrow\infty.$
\\*

\section{Further Integral formulas and applications}
In this section we write down some observations we made concerning different forms of integrals that may be evaluated. These integrals can be roughly categorized as belonging to the family of integrals in [1].
\begin{theorem} For $-1<\Re(s)<0,$
$$\int_{0}^{\infty}\frac{\{x\}^2x^{-s-1}}{x^2-[x]^2-\{x\}[x]}dx=-\frac{\zeta(s+1)}{s+1}. $$
\end{theorem}

\begin{proof} From [6, pg.14, eq.(2.1.5)], we have for $0<\Re(s)<1,$
\begin{equation}\begin{aligned}\frac{\zeta(s)}{s}&=-\int_{0}^{\infty}\{x\}x^{-s-1}dx\\
&=-\sum_{k\ge0}\int_{0}^{1}\frac{x}{(x+k)^{s+1}}dx. \end{aligned}\end{equation} A computation yields,
$$ \begin{aligned} \int_{0}^{\infty}\frac{\{x\}^2x^{-s-1}}{x^2-[x]^2-\{x\}[x]}dx &=\sum_{k\ge0}\int_{k}^{k+1}\frac{\{x\}^2x^{-s-1}}{x^2-[x]^2-\{x\}[x]}dx\\
&=\sum_{k\ge0}\int_{0}^{1}\frac{x^2}{((x+k)^2-k^2-xk)(x+k)^{s+1}}dx  \\ 
&=\sum_{k\ge0}\int_{0}^{1}\frac{x^2}{(x^2+xk)(x+k)^{s+1}}dx  \\ 
&=\sum_{k\ge0}\int_{0}^{1}\frac{x}{(x+k)^{s+2}}dx=-\frac{\zeta(s+1)}{s+1}.  \\ \end{aligned}$$
Here in the last line we have employed (3.1), and have take into account the valid region for $s.$
\end{proof}
Using similar arguments, we can use (1.2) to obtain the following theorem.
\begin{theorem} For $\Re(s)>0,$ we have
$$\int_{0}^{\infty}\frac{\{x\}^2[x]x^{-s-1}}{x^2-[x]^2-\{x\}[x]}dx=\frac{1}{s(s+1)}\left(s\zeta(s+1)+(1-s)\zeta(s)\right).$$
\end{theorem}
Next we consider a series identity as an application of Theorem 3.1.
\begin{theorem} For $x>0,$ we have
$$-\sum_{n\ge1}\frac{\Lambda(n)}{n}\frac{\{\frac{n}{x}\}^2}{(\frac{n}{x})^2-[\frac{n}{x}]^2-\{\frac{n}{x}\}[\frac{n}{x}]}=1-2\gamma-\log(x)+\sum_{\rho}\frac{\zeta(2-\rho)}{(2-\rho)}x^{1-\rho}+\sum_{n\ge1}\frac{\zeta(2+2n)}{(2+2n)}x^{1-2n}.$$
\end{theorem}
\begin{proof}
From Theorem 4, we have for $-1<r'<0,$
$$\begin{aligned}\sum_{n\ge1}\frac{\Lambda(n)}{n}\frac{\{\frac{n}{x}\}^2}{(\frac{n}{x})^2-[\frac{n}{x}]^2-\{\frac{n}{x}\}[\frac{n}{x}]} &=-\frac{1}{2\pi i}\int_{(r')}\frac{\zeta'(1-s)\zeta(s+1)}{\zeta(1-s)(s+1)}x^sds \\
&=-\frac{1}{2\pi i}\int_{(1-r')}\frac{\zeta'(s)\zeta(2-s)}{\zeta(s)(2-s)}x^{1-s}ds. \end{aligned}$$
The integrand has a double pole at $s=1,$ and simple poles at $s=\rho,$ and $s=-2n.$ For the first residue, we compute
$$\lim_{s\rightarrow1}\left(\frac{d}{ds}(s-1)^2\frac{\zeta'(s)\zeta(2-s)}{\zeta(s)(2-s)}x^{1-s}\right)  =1-2\gamma-\log(x).$$ The remaining computations are standard and therefore omitted.
\end{proof}
Note that the sums on the right hand side of Theorem 3.3 are absolutely convergent since $\zeta(s)\ll 1,$ when $\sigma>1,$ and so the formula is exact rather than asymptotic. 
\\*
{\bf Acknowledgement:} We thank Professor Ivi\'c for helpful comments.

1390 Bumps River Rd. \\*
Centerville, MA
02632 \\*
USA \\*
E-mail: alexpatk@hotmail.com, alexepatkowski@gmail.com

\begin{thebibliography}{9}


\bibitem{ConcreteMath} M. W. Coffey and M. C. Lettington, \emph{Mellin transforms with only critical zeros: Legendre functions,} J. of Number Theory, 148 (2015), 507--536.
\bibitem{ConcreteMath} P. Henrici, Applied and Computational Complex Analysis, Vol. 2, Wiley, New York, 1990.
\bibitem{ConcreteMath} J. Kaczorowski and K. Wiertelak, \emph{Oscillations of the remainder term related to the Euler totient function,} J. of Number Theory, 130 (2010), 2683--2700.
\bibitem{ConcreteMath} R. Kress, Linear Integral Equations, Springer, Berlin, 1989.
\bibitem{ConcreteMath} R. B. Paris, D. Kaminski, Asymptotics and Mellin--Barnes Integrals. Cambridge University Press. (2001)
\bibitem{ConcereteMath} E. C. Titchmarsh, \emph{The theory of the Riemann zeta function,} Oxford University Press,
2nd edition, 1986.
\bibitem{ConcreteMath} A. Walfisz, Weylsche Exponentialsummen in der neueren Zahlentheorie, VEB Deutcher Verlag der Wiss., Berlin, 1963.
\end{thebibliography}
\end{document}